\documentclass{amsart}
\usepackage{amsfonts}

\newtheorem{theorem}{Theorem}
\theoremstyle{plain}

\newtheorem{definition}{Definition}

\newtheorem{lemma}{Lemma}

\newtheorem{remark}{Remark}

\numberwithin{equation}{section}
\input{tcilatex}

\begin{document}
\title[Hermite-Hadamard Inequality]{Hermite-Hadamard's inequalities for
preinvex functions via fractional integrals and related fractional
inequalities}
\author{\.{I}mdat \.{I}\c{s}can}
\address{Department of Mathematics, Faculty of Science and Arts, Giresun
University, Giresun, Turkey}
\email{imdat.iscan@giresun.edu.tr}
\date{March 24, 2012}

\begin{abstract}
In this paper, first we have established Hermite- Hadamard's inequalities
for preinvex functions via fractional integrals. Second we extend some
estimates of the right side of a Hermite- Hadamard type inequality for
preinvex functions via fractional integrals.
\end{abstract}

\maketitle

\section{Introduction and Preliminaries}

Let $f:I\subset \mathbb{R\rightarrow R}$ be a convex mapping defined on the
interval $I$ of real numbers and $a,b\in I$ with $a<b$, then

\begin{equation}
f\left( \frac{a+b}{2}\right) \leq \frac{1}{b-a}\dint\limits_{a}^{b}f(x)dx%
\leq \frac{f(a)+f(b)}{2}\text{.}  \label{1-1}
\end{equation}

This doubly inequality is known in the literature as Hermite-Hadamard
integral inequality for convex mapping.We note that Hadamard's inequality
may be regarded as a refinement of the concept of convexity and it follows
easily from Jensen's inequality . For several recent results concerning the
inequality (\ref{1-1}) we refer the interested reader to\ \cite%
{DP00,BOP08,SSO10,KBOP07,PP00} and the references cited therein.

\begin{definition}
The function $f:\left[ a,b\right] \subset \mathbb{R\rightarrow R}$ is said
to be convex if the following inequality holds:%
\begin{equation*}
f\left( tx+(1-t)y\right) \leq tf(x)+(1-t)f(y)
\end{equation*}%
for all $x,y\in \left[ a,b\right] $ and $t\in \left[ 0,1\right] .$ We say
that f is concave if $\left( -f\right) $ is convex.
\end{definition}

In \cite{PP00} Pearce and Pe\v{c}ari\'{c} established the following result
connected with teh right part of (\ref{1-1}).

\begin{theorem}
\label{1.a}Let $f:I^{\circ }\subset \mathbb{R\rightarrow R}$ be a
differentiable mapping on $I^{\circ }$, $a,b\in I^{\circ }$ with $a<b$, and
let $q\geq 1.$ If the mapping $\left\vert f^{\prime }\right\vert ^{q}$
convex on $\left[ a,b\right] $, then 
\begin{equation}
\left\vert \frac{f(a)+f(b)}{2}-\frac{1}{b-a}\dint\limits_{a}^{b}f(x)dx\right%
\vert \leq \frac{b-a}{4}\left[ \frac{\left\vert f^{\prime }(a)\right\vert
^{q}+\left\vert f^{\prime }(b)\right\vert ^{q}}{2}\right] ^{\frac{1}{q}}
\label{0.1}
\end{equation}
\end{theorem}

The classical Hermite- Hadamard inequality provides estimates of the mean
value of a continuous convex function $f:\left[ a,b\right] \mathbb{%
\rightarrow R}$.

We give some necessary definitions and mathematical preliminaries of
fractional calculus theory which are used throughout this paper.

\begin{definition}
Let $f\in L\left[ a,b\right] $. The Riemann-Liouville integrals $%
J_{a^{+}}^{\alpha }f$ and $J_{b^{-}}^{\alpha }f$ of oder $\alpha >0$ with $%
a\geq 0$ are defined by

\begin{equation*}
J_{a^{+}}^{\alpha }f(x)=\frac{1}{\Gamma (\alpha )}\dint\limits_{a}^{x}\left(
x-t\right) ^{\alpha -1}f(t)dt,\ x>a
\end{equation*}

and

\begin{equation*}
J_{b^{-}}^{\alpha }f(x)=\frac{1}{\Gamma (\alpha )}\dint\limits_{x}^{b}\left(
t-x\right) ^{\alpha -1}f(t)dt,\ x<b
\end{equation*}

respectively, where $\Gamma (\alpha )$ is the Gamma function and $%
J_{a^{+}}^{0}f(x)=J_{b^{-}}^{0}f(x)=f(x).$
\end{definition}

In the case of $\alpha =1$, the fractional integral reduces to the classical
integral. Properties concerning this operator can be found (\cite{GM97}-\cite%
{P99}).

For some recent result connected with fractional integral see (\cite{SSYB11}-%
\cite{SO12}).

In \cite{SSYB11} Sar\i kaya et al. proved the following Hadamard type
inequalities for fractional integrals as follows.

\begin{theorem}
\label{1.1}Let $f:\left[ a,b\right] \mathbb{\rightarrow R}$ be a positive
function with $0\leq a<b$ and $f\in L\left[ a,b\right] .$ If $f$ is a convex
function on $\left[ a,b\right] $, then the following inequalities for
fractional integrals hold:%
\begin{equation}
f\left( \frac{a+b}{2}\right) \leq \frac{\Gamma (\alpha +1)}{2\left(
b-a\right) ^{\alpha }}\left[ J_{a^{+}}^{\alpha }f(x)+J_{b^{-}}^{\alpha }f(x)%
\right] \leq \frac{f(a)+f(b)}{2}  \label{1-2}
\end{equation}

with $\alpha >0.$
\end{theorem}

Using the following identity Sa\i kaya et al. in \cite{SSYB11} established
the following result which hold for convex functions.

\begin{lemma}
\label{1.11}Let $f:\left[ a,b\right] \mathbb{\rightarrow R}$ be a
differentiable mapping on $\left( a,b\right) $ with $a<b$. If $f^{\prime
}\in L\left[ a,b\right] $, then the following equality for fractional
integrals holds:%
\begin{equation}
\frac{f(a)+f(b)}{2}-\frac{\Gamma (\alpha +1)}{2\left( b-a\right) ^{\alpha }}%
\left[ J_{a^{+}}^{\alpha }f(x)+J_{b^{-}}^{\alpha }f(x)\right] =\frac{b-a}{2}%
\dint\limits_{0}^{1}\left[ \left( 1-t\right) ^{\alpha }-t^{\alpha }\right]
f^{\prime }\left( ta+(1-t)b\right) dt  \label{1-3}
\end{equation}
\end{lemma}

\begin{theorem}
\label{1.111}Let $f:\left[ a,b\right] \mathbb{\rightarrow R}$ be a
differentiable mapping on $\left( a,b\right) $ with $a<b$. If $\ \left\vert
f^{\prime }\right\vert $ is a convex function on $\left[ a,b\right] $, then
the following inequalities for fractional integrals holds:%
\begin{equation}
\left\vert \frac{f(a)+f(b)}{2}-\frac{\Gamma (\alpha +1)}{2\left( b-a\right)
^{\alpha }}\left[ J_{a^{+}}^{\alpha }f(x)+J_{b^{-}}^{\alpha }f(x)\right]
\right\vert \leq \frac{b-a}{2\left( \alpha +1\right) }\left( 1-\frac{1}{%
2^{\alpha }}\right) \left[ \left\vert f^{\prime }(a)\right\vert +\left\vert
f^{\prime }(b)\right\vert \right]  \label{1-4}
\end{equation}
\end{theorem}

In recent years several extentions and generalizations have been considered
for classical convexity. A significant generalization of convex functions is
that of invex functions introduced by Hanson in \cite{H81}. Weir and Mond 
\cite{WM98} introduced the concept of preinvex functions and applied it to
the establisment of the sufficient optimality conditions and duality in
nonlinear programming. Pini \cite{P91} introduced the concept of
prequasiinvex as a generalization of invex functions. Later, Mohan and Neogy 
\cite{MN95} obtained some properties of generalized preinvex functions. Noor 
\cite{N09}-\cite{N07} has established some Hermite-Hadamard type
inequalities for preinvex and log-preinvex functions.In recent papers Barani
et al. in  \cite{BGD11b} presented some estimates of the right hand side of
a Hermite-Hadamard type inequality in which some preinvex functions are
involved.

In this paper we generalized the results in \cite{BGD11b} and \cite{SSYB11}
for preinvex functions via fractional integrals. Now we recall some notions
in invexity analysis which will be used throught the paper (see \cite%
{A05,YL01} and references therein)

Let $f:A\mathbb{\rightarrow R}$ and $\eta :A\times A\rightarrow 
\mathbb{R}
,$where $A$ is a nonempty set in $%
\mathbb{R}
^{n}$, be continuous functions.

\begin{definition}
The set $A\subseteq $ $%
\mathbb{R}
^{n}$ is said to be invex with respect to $\eta (.,.)$, if for every $x,y\in
A$ and $t\in \left[ 0,1\right] ,$%
\begin{equation*}
x+t\eta (y,x)\in A.
\end{equation*}

The invex set $A$ is also called a $\eta -$connected set.
\end{definition}

It is obvious that every convex set is invex with respect to $\eta (y,x)=y-x$%
, but there exist invex sets which are not convex \cite{A05}.

\begin{definition}
The function $f$ on the invex set $A$ is said to be preinvex with respect to 
$\eta $ if 
\begin{equation*}
f\left( x+t\eta (y,x)\right) \leq \left( 1-t\right) f(x)+tf(y),\ \forall
x,y\in A,\ t\in \left[ 0,1\right] .
\end{equation*}

The function $f$ is said to be preconcave if and only if $-f$ \ is preinvex.
\end{definition}

Mohan and Neogy \cite{MN95} introduced condition C defined as follows

\textbf{Condition C:} Let $A\subseteq $ $%
\mathbb{R}
^{n}$ be an open invex subset with respect to $\eta :A\times A\rightarrow 
\mathbb{R}
.$ We say that the function $\eta $ satisfies the condition C if for any $%
x,y\in A$ and any $t\in \left[ 0,1\right] ,$%
\begin{eqnarray*}
\eta \left( y,y+t\eta (x,y)\right) &=&-t\eta (x,y) \\
\eta \left( x,y+t\eta (x,y)\right) &=&(1-t)\eta (x,y).
\end{eqnarray*}

Note that for every $x,y\in A$ and every $t\in \left[ 0,1\right] $ from
condition C, we have

\begin{equation}
\eta \left( y+t_{2}\eta (x,y),y+t_{1}\eta (x,y)\right) =(t_{2}-t_{1})\eta
(x,y).  \label{1-5}
\end{equation}%
We will use the condition in our main results.

In \cite{N07b} Noor proved the Hermite-Hadamard inequality for the preinvex
functions as follows:

\begin{theorem}
\label{1.2}Let $f:K=\left[ a,a+\eta (b,a)\right] \rightarrow \left( 0,\infty
\right) $ be a preinvex function on the interval of real numbers $K^{o}$
(the interior of K) and $a,b\in K^{o}$ with $a<a+\eta (b,a)$. Then the
following inequality holds:%
\begin{equation}
f\left( \frac{2a+\eta (b,a)}{2}\right) \leq \frac{1}{\eta (b,a)}%
\dint\limits_{a}^{a+\eta (b,a)}f(x)dx\leq \frac{f(a)+f(b)}{2}  \label{1-55}
\end{equation}
\end{theorem}

In \cite{BGD11b} Barani, Gahazanfari, and Dragomir proved the following
theorems:

\begin{theorem}
\label{1.3}Let $A\subseteq $ $%
\mathbb{R}
$ be an open invex subset with respect to $\eta :A\times A\rightarrow 
\mathbb{R}
.$ Suppose that $f:A\rightarrow 
\mathbb{R}
$ is a differentiable function. If $\left\vert f^{\prime }\right\vert $ is
preinvex on $A$ then, for every $a,b\in A$ with $\eta (b,a)\neq 0$ the
following inequalities holds%
\begin{eqnarray}
&&\left\vert \frac{f(a)+f\left( a+\eta (b,a)\right) }{2}-\frac{1}{\eta (b,a)}%
\dint\limits_{a}^{a+\eta (b,a)}f(x)dx\right\vert  \notag \\
&\leq &\frac{\left\vert \eta (b,a)\right\vert }{8}\left[ \left\vert
f^{\prime }(a)\right\vert +\left\vert f^{\prime }(b)\right\vert \right] .
\label{1-6}
\end{eqnarray}
\end{theorem}

\begin{theorem}
\label{1.b}Let $A\subseteq $ $%
\mathbb{R}
$ be an open invex subset with respect to $\eta :A\times A\rightarrow 
\mathbb{R}
.$ Suppose that $f:A\rightarrow 
\mathbb{R}
$ is a differentiable function. Assume that $p\in 
\mathbb{R}
$ with $p>1.$ If $\left\vert f^{\prime }\right\vert ^{\frac{p}{p-1}}$ is
preinvex on $A$ then, for every $a,b\in A$ with $\eta (b,a)\neq 0$ the
following inequalities holds 
\begin{eqnarray}
&&\left\vert \frac{f(a)+f\left( a+\eta (b,a)\right) }{2}-\frac{1}{\eta (b,a)}%
\dint\limits_{a}^{a+\eta (b,a)}f(x)dx\right\vert  \notag \\
&\leq &\frac{\left\vert \eta (b,a)\right\vert }{2(p+1)^{\frac{1}{p}}}\left[
\left\vert f^{\prime }(a)\right\vert ^{\frac{p}{p-1}}+\left\vert f^{\prime
}(b)\right\vert ^{^{\frac{p}{p-1}}}\right] ^{^{\frac{p-1}{p}}}.  \label{1-7}
\end{eqnarray}
\end{theorem}

\section{Hermite-Hadamard type inequalities for preinvex functions via
fractional integrals}

\begin{theorem}
\label{2.1}Let $A\subseteq $ $%
\mathbb{R}
$ be an open invex subset with respect to $\eta :A\times A\rightarrow 
\mathbb{R}
$ and $a,b\in A$ with $a<a+\eta (b,a).$ If $f:\left[ a,a+\eta (b,a)\right]
\rightarrow \left( 0,\infty \right) $ is a preinvex function, $f\in L\left[
a,a+\eta (b,a)\right] $ and $\eta $ satisfies condition C then, the
following inequalities for fractional integrals holds:%
\begin{eqnarray}
f\left( \frac{2a+\eta (b,a)}{2}\right) &\leq &\frac{\Gamma (\alpha +1)}{%
2\eta ^{\alpha }(b,a)}\left[ J_{a^{+}}^{\alpha }f(a+\eta (b,a))+J_{\left(
a+\eta (b,a)\right) ^{-}}^{\alpha }f(a)\right]  \notag \\
&\leq &\frac{f(a)+f(a+\eta (b,a))}{2}\leq \frac{f(a)+f(b)}{2}  \label{2}
\end{eqnarray}

with $\alpha >0.$
\end{theorem}

\begin{proof}
Since $a,b\in A$ and $A$ is an invex set with respect to $\eta $, for every $%
t\in \left[ 0,1\right] $, we have $a+t\eta (b,a)\in A.$ By preinvexity of $f$%
, we have for every $x,y\in \left[ a,a+\eta (b,a)\right] $ with $t=\frac{1}{2%
}$%
\begin{equation*}
f\left( x+\frac{\eta (y,x)}{2}\right) \leq \frac{f(x)+f(y)}{2}
\end{equation*}

i.e. with $x=a+(1-t)\eta (b,a),\ \ y=a+t\eta (b,a)$ from inequality (\ref%
{1-5}) we get

\begin{eqnarray}
&&2f\left( a+(1-t)\eta (b,a)+\frac{\eta (a+t\eta (b,a),a+(1-t)\eta (b,a))}{2}%
\right)  \notag \\
&=&2f\left( a+(1-t)\eta (b,a)+\frac{(2t-1)\eta (b,a))}{2}\right) =2f\left( 
\frac{2a+\eta (b,a)}{2}\right)  \notag \\
&\leq &f\left( a+(1-t)\eta (b,a)\right) +f\left( a+t\eta (b,a)\right)
\label{2-1}
\end{eqnarray}

Multiplying both sides (\ref{2-1}) by $t^{\alpha -1}$, then integrating the
resulting inequality with respect to $t$ over $\left[ 0,1\right] ,$ we obtain%
\begin{eqnarray*}
&&\frac{2}{\alpha }f\left( \frac{2a+\eta (b,a)}{2}\right) \\
&\leq &\dint\limits_{0}^{1}t^{\alpha -1}f\left( a+(1-t)\eta (b,a)\right)
dt+\dint\limits_{0}^{1}t^{\alpha -1}f\left( a+t\eta (b,a)\right) dt \\
&=&\frac{1}{\eta ^{\alpha }(b,a)}\left[ \dint\limits_{a}^{a+\eta
(b,a)}\left( a+\eta (b,a)-u\right) ^{\alpha
-1}f(u)du+\dint\limits_{a}^{a+\eta (b,a)}\left( u-a\right) ^{\alpha -1}f(u)du%
\right] \\
&=&\frac{\Gamma (\alpha )}{2\eta ^{\alpha }(b,a)}\left[ J_{a^{+}}^{\alpha
}f(a+\eta (b,a))+J_{\left( a+\eta (b,a)\right) ^{-}}^{\alpha }f(a)\right]
\end{eqnarray*}

i.e.

\begin{equation*}
f\left( \frac{2a+\eta (b,a)}{2}\right) \leq \frac{\Gamma (\alpha +1)}{2\eta
^{\alpha }(b,a)}\left[ J_{a^{+}}^{\alpha }f(a+\eta (b,a))+J_{a+\eta
(b,a)^{-}}^{\alpha }f(a)\right]
\end{equation*}

and the fist inequality is proved.

For the proof of the second inequality in (\ref{2-1}) we first note that if $%
f$ is a preinvex function on $\left[ a,a+\eta (b,a)\right] $ and the mapping 
$\eta $ satisfies condition C then for every $t\in \left[ 0,1\right] ,$ from
inequality (\ref{1-5}) it yields%
\begin{eqnarray}
f\left( a+t\eta (b,a)\right)  &=&f\left( a+\eta (b,a)+(1-t)\eta (a,a+\eta
(b,a))\right)   \notag \\
&\leq &tf\left( a+\eta (b,a)\right) +(1-t)f(a)  \label{2-11}
\end{eqnarray}

and similarly%
\begin{eqnarray*}
f\left( a+(1-t)\eta (b,a)\right) &=&f\left( a+\eta (b,a)+t\eta (a,a+\eta
(b,a))\right) \\
&\leq &(1-t)f\left( a+\eta (b,a)\right) +tf(a).
\end{eqnarray*}

By adding these inequalities we have

\begin{equation}
f\left( a+t\eta (b,a)\right) +f\left( a+(1-t)\eta (b,a)\right) \leq
f(a)+f\left( a+\eta (b,a)\right)  \label{2-2}
\end{equation}

Then multiplying both (\ref{2-2}) by $t^{\alpha -1}$ and integrating the
resulting inequality with respect to $t$ over $\left[ 0,1\right] ,$ we obtain%
\begin{equation*}
\dint\limits_{0}^{1}t^{\alpha -1}f\left( a+t\eta (b,a)\right)
dt+\dint\limits_{0}^{1}t^{\alpha -1}f\left( a+(1-t)\eta (b,a)\right) dt\leq 
\left[ f(a)+f\left( a+\eta (b,a)\right) \right] \dint\limits_{0}^{1}t^{%
\alpha -1}dt.
\end{equation*}

i.e.%
\begin{equation*}
\frac{\Gamma (\alpha )}{\eta ^{\alpha }(b,a)}\left[ J_{a^{+}}^{\alpha
}f(a+\eta (b,a))+J_{\left( a+\eta (b,a)\right) ^{-}}^{\alpha }f(a)\right]
\leq \frac{f(a)+f\left( a+\eta (b,a)\right) }{\alpha }.
\end{equation*}

Using the mapping $\eta $ satisfies condition C the proof is completed.
\end{proof}

\begin{remark}
a) If in Theorem \ref{2.1}, we let $\eta (b,a)=b-a$, then inequality (\ref{2}%
) become inequality (\ref{1-2}) of Theorem \ref{1.1}.

b) If in Theorem \ref{2.1}, we let $\alpha =1$, then inequality (\ref{2})
become inequality (\ref{1-55}) of Theorem \ref{1.2}.
\end{remark}

\begin{lemma}
\label{2.2}Let $A\subseteq $ $%
\mathbb{R}
$ be an open invex subset with respect to $\eta :A\times A\rightarrow 
\mathbb{R}
$ and $a,b\in A$ with $a<a+\eta (b,a).$ Suppose that $f:A\rightarrow 
\mathbb{R}
$ is a differentiable function. If $f^{\prime }$ is preinvex function on $A$
and $f^{\prime }\in L\left[ a,a+\eta (b,a)\right] $ then, the following
equality holds:%
\begin{eqnarray}
&&\frac{f(a)+f\left( a+\eta (b,a)\right) }{2}-\frac{\Gamma (\alpha +1)}{%
2\eta ^{\alpha }(b,a)}\left[ J_{a^{+}}^{\alpha }f(a+\eta (b,a))+J_{\left(
a+\eta (b,a)\right) ^{-}}^{\alpha }f(a)\right]  \notag \\
&=&\frac{\eta (b,a)}{2}\dint\limits_{0}^{1}\left[ t^{\alpha }-\left(
1-t\right) ^{\alpha }\right] f^{\prime }\left( a+t\eta (b,a)\right) dt
\label{2-3}
\end{eqnarray}
\end{lemma}

\begin{proof}
It suffices to note that 
\begin{eqnarray}
I &=&\dint\limits_{0}^{1}\left[ t^{\alpha }-\left( 1-t\right) ^{\alpha }%
\right] f^{\prime }\left( a+t\eta (b,a)\right) dt  \notag \\
&=&\left[ \dint\limits_{0}^{1}t^{\alpha }f^{\prime }\left( a+t\eta
(b,a)\right) dt\right] +\left[ -\dint\limits_{0}^{1}\left( 1-t\right)
^{\alpha }f^{\prime }\left( a+t\eta (b,a)\right) dt\right]  \notag \\
&&I_{1}+I_{2}  \label{2-4}
\end{eqnarray}

integrating by parts%
\begin{eqnarray}
I_{1} &=&\dint\limits_{0}^{1}t^{\alpha }f^{\prime }\left( a+t\eta
(b,a)\right) dt  \notag \\
&=&\left. t^{\alpha }\frac{f\left( a+t\eta (b,a)\right) }{\eta (b,a)}%
\right\vert _{0}^{1}-\dint\limits_{0}^{1}\alpha t^{\alpha -1}\frac{f\left(
a+t\eta (b,a)\right) }{\eta (b,a)}dt  \notag \\
&=&\frac{f\left( a+\eta (b,a)\right) }{\eta (b,a)}-\frac{\alpha }{\eta (b,a)}%
\dint\limits_{a}^{a+\eta (b,a)}\left( \frac{x-a}{\eta (b,a)}\right) ^{\alpha
-1}\frac{f\left( x\right) }{\eta (b,a)}dx  \notag \\
&=&\frac{f\left( a+\eta (b,a)\right) }{\eta (b,a)}-\frac{\Gamma (\alpha +1)}{%
\eta ^{\alpha +1}(b,a)}J_{\left( a+\eta (b,a)\right) ^{-}}^{\alpha }f(a)
\label{2-5}
\end{eqnarray}

and similarly we get,%
\begin{eqnarray}
I_{2} &=&-\dint\limits_{0}^{1}\left( 1-t\right) ^{\alpha }f^{\prime }\left(
a+t\eta (b,a)\right) dt  \notag \\
&=&\left. -\left( 1-t\right) ^{\alpha }\frac{f\left( a+t\eta (b,a)\right) }{%
\eta (b,a)}\right\vert _{0}^{1}-\dint\limits_{0}^{1}\alpha \left( 1-t\right)
^{\alpha -1}\frac{f\left( a+t\eta (b,a)\right) }{\eta (b,a)}dt  \notag \\
&=&\frac{f\left( a\right) }{\eta (b,a)}-\frac{\alpha }{\eta (b,a)}%
\dint\limits_{a}^{a+\eta (b,a)}\left( \frac{a+t\eta (b,a)-x}{\eta (b,a)}%
\right) ^{\alpha -1}\frac{f\left( x\right) }{\eta (b,a)}dx  \notag \\
&=&\frac{f\left( a\right) }{\eta (b,a)}-\frac{\Gamma (\alpha +1)}{\eta
^{\alpha +1}(b,a)}J_{a^{+}}^{\alpha }f(a+\eta (b,a))  \label{2-6}
\end{eqnarray}

Using (\ref{2-5}) and (\ref{2-6}) in (\ref{2-4}), it follows that%
\begin{equation*}
I=\frac{f(a)+f\left( a+\eta (b,a)\right) }{2}-\frac{\Gamma (\alpha +1)}{%
2\eta ^{\alpha }(b,a)}\left[ J_{a^{+}}^{\alpha }f(a+\eta (b,a))+J_{\left(
a+\eta (b,a)\right) ^{-}}^{\alpha }f(a)\right] .
\end{equation*}

Thus, by multiplying both sides by $\frac{\eta (b,a)}{2}$, we have
conclusion (\ref{2-3}).
\end{proof}

\begin{remark}
If in Lemma \ref{2.2}, we let $\eta (b,a)=b-a$, then equality (\ref{2-3})
become inequality (\ref{1-3}) of Lemma \ref{1.11}.
\end{remark}

\begin{theorem}
\label{2.3}Let $A\subseteq $ $%
\mathbb{R}
$ be an open invex subset with respect to $\eta :A\times A\rightarrow 
\mathbb{R}
$ and $a,b\in A$ with $a<a+\eta (b,a)$ such that $f^{\prime }\in L\left[
a,a+\eta (b,a)\right] $. Suppose that $f:A\rightarrow 
\mathbb{R}
$ is a differentiable function. If $\left\vert f^{\prime }\right\vert $ is
preinvex function on $A$ then the following inequality for fractional
integrals with $\alpha >0$ holds: 
\begin{eqnarray}
&&\left\vert \frac{f(a)+f\left( a+\eta (b,a)\right) }{2}-\frac{\Gamma
(\alpha +1)}{2\eta ^{\alpha }(b,a)}\left[ J_{a^{+}}^{\alpha }f(a+\eta
(b,a))+J_{\left( a+\eta (b,a)\right) ^{-}}^{\alpha }f(a)\right] \right\vert
\label{2-7} \\
&\leq &\frac{\eta (b,a)}{2\left( \alpha +1\right) }\left( 1-\frac{1}{%
2^{\alpha }}\right) \left[ \left\vert f^{\prime }(a)\right\vert +\left\vert
f^{\prime }(b)\right\vert \right]  \notag
\end{eqnarray}
\end{theorem}

\begin{proof}
Using lemma \ref{2.2} and the preinvexity of $\left\vert f^{\prime
}\right\vert $ we get%
\begin{eqnarray*}
&&\left\vert \frac{f(a)+f\left( a+\eta (b,a)\right) }{2}-\frac{\Gamma
(\alpha +1)}{2\eta ^{\alpha }(b,a)}\left[ J_{a^{+}}^{\alpha }f(a+\eta
(b,a))+J_{\left( a+\eta (b,a)\right) ^{-}}^{\alpha }f(a)\right] \right\vert
\\
&\leq &\frac{\eta (b,a)}{2}\dint\limits_{0}^{1}\left\vert t^{\alpha }-\left(
1-t\right) ^{\alpha }\right\vert \left\vert f^{\prime }\left( a+t\eta
(b,a)\right) \right\vert dt \\
&\leq &\frac{\eta (b,a)}{2}\dint\limits_{0}^{1}\left\vert t^{\alpha }-\left(
1-t\right) ^{\alpha }\right\vert \left[ (1-t)\left\vert f^{\prime }\left(
a\right) \right\vert +t\left\vert f^{\prime }\left( b\right) \right\vert %
\right] dt \\
&\leq &\frac{\eta (b,a)}{2}\left\{ \dint\limits_{0}^{\frac{1}{2}}\left[
\left( 1-t\right) ^{\alpha }-t^{\alpha }\right] \left[ (1-t)\left\vert
f^{\prime }\left( a\right) \right\vert +t\left\vert f^{\prime }\left(
b\right) \right\vert \right] dt+\dint\limits_{\frac{1}{2}}^{1}\left[
t^{\alpha }-\left( 1-t\right) ^{\alpha }\right] \left[ (1-t)\left\vert
f^{\prime }\left( a\right) \right\vert +t\left\vert f^{\prime }\left(
b\right) \right\vert \right] dt\right\} \\
&=&\frac{\eta (b,a)}{2}\left[ \left\vert f^{\prime }\left( a\right)
\right\vert +\left\vert f^{\prime }\left( b\right) \right\vert \right]
\left( \dint\limits_{0}^{\frac{1}{2}}\left[ \left( 1-t\right) ^{\alpha
}-t^{\alpha }\right] dt\right) \\
&=&\frac{\eta (b,a)}{2\left( \alpha +1\right) }\left( 1-\frac{1}{2^{\alpha }}%
\right) \left[ \left\vert f^{\prime }(a)\right\vert +\left\vert f^{\prime
}(b)\right\vert \right] ,
\end{eqnarray*}%
which completes the proof.
\end{proof}

\begin{remark}
a) If in Theorem \ref{2.3}, we let $\eta (b,a)=b-a$, then inequality (\ref%
{2-7}) become inequality (\ref{1-4}) of Theorem \ref{1.111}.

b) If in Theorem \ref{2.3}, we let $\alpha =1$, then inequality (\ref{2-7})
become inequality (\ref{1-6}) of Theorem \ref{1.3}.

c) In Theorem \ref{2.3}, assume that $\eta $ satisfies condition C and using
inequality (\ref{2-11}) we get%
\begin{eqnarray*}
&&\left\vert \frac{f(a)+f\left( a+\eta (b,a)\right) }{2}-\frac{\Gamma
(\alpha +1)}{2\eta ^{\alpha }(b,a)}\left[ J_{a^{+}}^{\alpha }f(a+\eta
(b,a))+J_{\left( a+\eta (b,a)\right) ^{-}}^{\alpha }f(a)\right] \right\vert
\\
&\leq &\frac{\eta (b,a)}{2\left( \alpha +1\right) }\left( 1-\frac{1}{%
2^{\alpha }}\right) \left[ \left\vert f^{\prime }(a)\right\vert +\left\vert
f^{\prime }(a+\eta (b,a))\right\vert \right]
\end{eqnarray*}

\begin{theorem}
\label{2.4}Let $A\subseteq $ $%
\mathbb{R}
$ be an open invex subset with respect to $\eta :A\times A\rightarrow 
\mathbb{R}
$ and $a,b\in A$ with $a<a+\eta (b,a)$ such that $f^{\prime }\in L\left[
a,a+\eta (b,a)\right] $. Suppose that $f:A\rightarrow 
\mathbb{R}
$ is a differentiable function. If $\left\vert f^{\prime }\right\vert ^{q}$
is preinvex function on $A$ for some fixed $q>1$ then the following
inequality holds:%
\begin{eqnarray}
&&\left\vert \frac{f(a)+f\left( a+\eta (b,a)\right) }{2}-\frac{\Gamma
(\alpha +1)}{2\eta ^{\alpha }(b,a)}\left[ J_{a^{+}}^{\alpha }f(a+\eta
(b,a))+J_{\left( a+\eta (b,a)\right) ^{-}}^{\alpha }f(a)\right] \right\vert
\label{2-8} \\
&\leq &\frac{\eta (b,a)}{2\left( \alpha p+1\right) ^{\frac{1}{p}}}\left( 
\frac{\left\vert f^{\prime }(a)\right\vert ^{q}+\left\vert f^{\prime
}(b)\right\vert ^{q}}{2}\right) ^{\frac{1}{q}}  \notag
\end{eqnarray}%
where $\frac{1}{p}+\frac{1}{q}=1$ and $\alpha \in \left[ 0,1\right] .$
\end{theorem}
\end{remark}

\begin{proof}
From lemma\ref{2.2} and using H\"{o}lder inequality with properties of
modulus, we have%
\begin{eqnarray*}
&&\left\vert \frac{f(a)+f\left( a+\eta (b,a)\right) }{2}-\frac{\Gamma
(\alpha +1)}{2\eta ^{\alpha }(b,a)}\left[ J_{a^{+}}^{\alpha }f(a+\eta
(b,a))+J_{\left( a+\eta (b,a)\right) ^{-}}^{\alpha }f(a)\right] \right\vert
\\
&\leq &\frac{\eta (b,a)}{2}\dint\limits_{0}^{1}\left\vert t^{\alpha }-\left(
1-t\right) ^{\alpha }\right\vert \left\vert f^{\prime }\left( a+t\eta
(b,a)\right) \right\vert dt \\
&\leq &\frac{\eta (b,a)}{2}\left( \dint\limits_{0}^{1}\left\vert t^{\alpha
}-\left( 1-t\right) ^{\alpha }\right\vert ^{p}dt\right) ^{\frac{1}{p}}\left(
\dint\limits_{0}^{1}\left\vert f^{\prime }\left( a+t\eta (b,a)\right)
\right\vert ^{q}dt\right) ^{\frac{1}{q}}.
\end{eqnarray*}%
We know that for $\alpha \in \left[ 0,1\right] $ and $\forall t_{1},t_{2}\in %
\left[ 0,1\right] $, 
\begin{equation*}
\left\vert t_{1}^{\alpha }-t_{2}^{\alpha }\right\vert \leq \left\vert
t_{1}-t_{2}\right\vert ^{\alpha },
\end{equation*}%
therefore%
\begin{eqnarray*}
\dint\limits_{0}^{1}\left\vert t^{\alpha }-\left( 1-t\right) ^{\alpha
}\right\vert ^{p}dt &\leq &\dint\limits_{0}^{1}\left\vert 1-2t\right\vert
^{\alpha p}dt \\
&=&\dint\limits_{0}^{\frac{1}{2}}\left[ 1-2t\right] ^{\alpha
p}dt+\dint\limits_{\frac{1}{2}}^{1}\left[ 2t-1\right] ^{\alpha p}dt \\
&=&\frac{1}{\alpha p+1}.
\end{eqnarray*}%
Since $\left\vert f^{\prime }\right\vert ^{q}$ is convex on $\left[ a,a+\eta
(b,a)\right] ,$ we have inequality (\ref{2-8}), which completes the proof.
\end{proof}

\begin{remark}
a) If in Theorem \ref{2.4}, we let $\eta (b,a)=b-a$ and $\alpha =1$then
inequality (\ref{2-8}) become inequality (\ref{1-7}) of Theorem\ref{1.b}.

b) In Theorem \ref{2.4}, assume that $\eta $ satisfies condition C and using
inequality (\ref{2-11}) we get%
\begin{eqnarray*}
&&\left\vert \frac{f(a)+f\left( a+\eta (b,a)\right) }{2}-\frac{\Gamma
(\alpha +1)}{2\eta ^{\alpha }(b,a)}\left[ J_{a^{+}}^{\alpha }f(a+\eta
(b,a))+J_{\left( a+\eta (b,a)\right) ^{-}}^{\alpha }f(a)\right] \right\vert
\\
&\leq &\frac{\eta (b,a)}{2\left( \alpha p+1\right) ^{\frac{1}{p}}}\left( 
\frac{\left\vert f^{\prime }(a)\right\vert ^{q}+\left\vert f^{\prime
}(a+\eta (b,a))\right\vert ^{q}}{2}\right) ^{\frac{1}{q}}.
\end{eqnarray*}
\end{remark}

\begin{theorem}
\label{2.5}Let $A\subseteq $ $%
\mathbb{R}
$ be an open invex subset with respect to $\eta :A\times A\rightarrow 
\mathbb{R}
$ and $a,b\in A$ with $a<a+\eta (b,a)$ such that $f^{\prime }\in L\left[
a,a+\eta (b,a)\right] $. Suppose that $f:A\rightarrow 
\mathbb{R}
$ is a differentiable function. If $\left\vert f^{\prime }\right\vert ^{q}$
is preinvex function on $A$ for some fixed $q>1$ then the following
inequality holds:%
\begin{eqnarray}
&&\left\vert \frac{f(a)+f\left( a+\eta (b,a)\right) }{2}-\frac{\Gamma
(\alpha +1)}{2\eta ^{\alpha }(b,a)}\left[ J_{a^{+}}^{\alpha }f(a+\eta
(b,a))+J_{\left( a+\eta (b,a)\right) ^{-}}^{\alpha }f(a)\right] \right\vert
\label{2-9} \\
&\leq &\frac{\eta (b,a)}{\left( \alpha +1\right) }\left( 1-\frac{1}{%
2^{\alpha }}\right) \left[ \frac{\left\vert f^{\prime }(a)\right\vert
^{q}+\left\vert f^{\prime }(b)\right\vert ^{q}}{2}\right] ^{\frac{1}{q}} 
\notag
\end{eqnarray}%
where $\frac{1}{p}+\frac{1}{q}=1$ and $\alpha >0.$
\end{theorem}

\begin{proof}
From lemma\ref{2.2} and using H\"{o}lder inequality with properties of
modulus, we have%
\begin{eqnarray*}
&&\left\vert \frac{f(a)+f\left( a+\eta (b,a)\right) }{2}-\frac{\Gamma
(\alpha +1)}{2\eta ^{\alpha }(b,a)}\left[ J_{a^{+}}^{\alpha }f(a+\eta
(b,a))+J_{\left( a+\eta (b,a)\right) ^{-}}^{\alpha }f(a)\right] \right\vert
\\
&\leq &\frac{\eta (b,a)}{2}\dint\limits_{0}^{1}\left\vert t^{\alpha }-\left(
1-t\right) ^{\alpha }\right\vert ^{\frac{1}{p}+\frac{1}{q}}\left\vert
f^{\prime }\left( a+t\eta (b,a)\right) \right\vert dt \\
&\leq &\frac{\eta (b,a)}{2}\left( \dint\limits_{0}^{1}\left\vert t^{\alpha
}-\left( 1-t\right) ^{\alpha }\right\vert dt\right) ^{\frac{1}{p}}\left(
\dint\limits_{0}^{1}\left\vert t^{\alpha }-\left( 1-t\right) ^{\alpha
}\right\vert \left\vert f^{\prime }\left( a+t\eta (b,a)\right) \right\vert
^{q}dt\right) ^{\frac{1}{q}}.
\end{eqnarray*}%
On the other hand, we have%
\begin{eqnarray*}
\dint\limits_{0}^{1}\left\vert t^{\alpha }-\left( 1-t\right) ^{\alpha
}\right\vert dt &=&\dint\limits_{0}^{\frac{1}{2}}\left[ \left( 1-t\right)
^{\alpha }-t^{\alpha }\right] dt+\dint\limits_{\frac{1}{2}}^{1}\left[
t^{\alpha }-\left( 1-t\right) ^{\alpha }\right] dt \\
&=&\frac{2}{\alpha +1}\left( 1-\frac{1}{2^{\alpha }}\right) .
\end{eqnarray*}%
Since $\left\vert f^{\prime }\right\vert ^{q}$ is preinvex function on $A$,
we obtain%
\begin{equation*}
\left\vert f^{\prime }\left( a+t\eta (b,a)\right) \right\vert ^{q}\leq
(1-t)\left\vert f^{\prime }(a)\right\vert ^{q}+t\left\vert f^{\prime
}(b)\right\vert ^{q},\ \ t\in \left[ 0,1\right]
\end{equation*}%
and 
\begin{eqnarray*}
\dint\limits_{0}^{1}\left\vert t^{\alpha }-\left( 1-t\right) ^{\alpha
}\right\vert \left\vert f^{\prime }\left( a+t\eta (b,a)\right) \right\vert
^{q}dt &\leq &\dint\limits_{0}^{1}\left\vert t^{\alpha }-\left( 1-t\right)
^{\alpha }\right\vert \left[ (1-t)\left\vert f^{\prime }(a)\right\vert
^{q}+t\left\vert f^{\prime }(b)\right\vert ^{q}\right] dt \\
&=&\dint\limits_{0}^{\frac{1}{2}}\left[ \left( 1-t\right) ^{\alpha
}-t^{\alpha }\right] \left[ (1-t)\left\vert f^{\prime }(a)\right\vert
^{q}+t\left\vert f^{\prime }(b)\right\vert ^{q}\right] dt \\
&&+\dint\limits_{\frac{1}{2}}^{1}\left[ t^{\alpha }-\left( 1-t\right)
^{\alpha }\right] \left[ (1-t)\left\vert f^{\prime }(a)\right\vert
^{q}+t\left\vert f^{\prime }(b)\right\vert ^{q}\right] dt \\
&=&\frac{1}{\alpha +1}\left( 1-\frac{1}{2^{\alpha }}\right) \left[
\left\vert f^{\prime }(a)\right\vert ^{q}+\left\vert f^{\prime
}(b)\right\vert ^{q}\right]
\end{eqnarray*}

from here we obtain inequality (\ref{2-9}) which completes the proof.
\end{proof}

\begin{remark}
a) If in Theorem\ref{2.5}, we let $\eta (b,a)=b-a$ and $\alpha =1$ then
inequality (\ref{2-9})become inequality (\ref{0.1}) Theorem\ref{1.a}.

b) In Theorem\ref{2.5}, assume that $\eta $ satisfies condition C and using
inequality (\ref{2-11}) we get%
\begin{eqnarray*}
&&\left\vert \frac{f(a)+f\left( a+\eta (b,a)\right) }{2}-\frac{\Gamma
(\alpha +1)}{2\eta ^{\alpha }(b,a)}\left[ J_{a^{+}}^{\alpha }f(a+\eta
(b,a))+J_{\left( a+\eta (b,a)\right) ^{-}}^{\alpha }f(a)\right] \right\vert
\\
&\leq &\frac{\eta (b,a)}{\left( \alpha +1\right) }\left( 1-\frac{1}{%
2^{\alpha }}\right) \left[ \frac{\left\vert f^{\prime }(a)\right\vert
^{q}+\left\vert f^{\prime }(a+\eta (b,a))\right\vert ^{q}}{2}\right] ^{\frac{%
1}{q}}.
\end{eqnarray*}
\end{remark}


\begin{thebibliography}{99}
\bibitem{DP00} S.S. Dragomir and C.E.M. Pearce, Selected Topics on
Hermite-Hadamard Inequalities and Applications, RGMIA Monographs,Victoria
University, 2000.

\bibitem{BOP08} M.K. Bakula, M.E. Ozdemir, J. Pe\v{c}ari\'{c}, Hadamard type
inequalities for $m$-convex and $(\alpha ,m)$-convex functions, J. Inequal.
Pure Appl. Math. 9 (2008) Article 96. [Online: http://jipam.vu.edu.au].

\bibitem{SSO10} M.Z. Sar\i kaya, E. Set, M.E. \"{O}zdemir, On some new
inequalities of Hadamard type involving h-convex functions, Acta Nath. Univ.
Comenianae vol. LXXIX, 2 (2010),pp. 265-272.

\bibitem{KBOP07} U.S. K\i rmac\i , M.K. Bakula, M.E. Ozdemir, J. Pecaric,
Hadamard's type inequalities for $s$-convex functions, Appl. Math. Comp.,
193 (2007), 26-35.

\bibitem{PP00} C.E.M. Pearce and J. Pe\v{c}ari\'{c}, Inequalities for
diffrentiable mapping with application to special means and quadrature
formula. Appl. Math. Lett., 13 (2000), 51-55.

\bibitem{DA98} S.S. Dragomir and R.P. Agarwal, Two inequalities for
differentiable mappings and applications to special means of real numbers
and to trapezoidal formula, Appl. Math. Lett., 11 (1998) 91-95.

\bibitem{GM97} R. Gorenflo, F. Mainardi, Fractional calculus; integral and
differential equations of fractional order, Springer Verlag, Wien (1997),
223-276.

\bibitem{MR93} S. Miller and B. Ross, An introduction to the Fractional
Calculus and Fractional Differential Equations, John Wiley \& Sons, USA,
1993, 2.

\bibitem{P99} I. Podlubni, Fractional Differential Equations, Academic
Press, San Diego, 1999.

\bibitem{SSYB11} M.Z. Sar\i kaya, E. Set, H. Yald\i z and N. Ba\c{s}ak,
Hermite-Hadamard's inequalities for fractional integrals and related
fractional inequalities, Mathematical and Computer Modelling,
DOI:10.1016/j.mcm.2011.12.048.

\bibitem{D10} Z. Dahmani, On Minkowski and Hermite-Hadamard integral
inequalities via fractional via fractional integration, Ann. Funct. Anal. 1
(1) (2010), 51-58

\bibitem{S12} E. Set, New inequalities of Ostrowski type for mapping whose
derivatives are $s$-convex in the second sense via fractional integrals,
Computers and Math. with Appl. 63 (2012) 1147-1154.

\bibitem{SO12} M.Z. Sar\i kaya and H. Ogunmez, On new inequalities via
Riemann-Liouville fractional integration, arXive:1005.1167v1, submitted.

\bibitem{H81} M.A. Hanson, On sufficiency of the Kuhn-Tucker conditions, J.
Math. Anal. Appl. 80 (1981) 545-550.

\bibitem{WM98} T. Weir, and B. Mond, Preinvex functions in multiple
objective optimization, Journal of Mathematical Analysis and Applications,
136, (1198) 29-38.

\bibitem{P91} R. Pini, Invexity and generalized Convexity, Optimization 22
(1991) 513-525.

\bibitem{N09} M. Aslam Noor, Hadamard integral inequalities for product of
two preinvex function, Nonl. anal. Forum, 14 (2009), 167-173.

\bibitem{N06} M. Aslam Noor, Some new classes of nonconvex functionss, Nonl.
Funct. Anal. Appl., 11 (2006), 165-171.

\bibitem{N07} M. Aslam Noor, On Hadamard integral inequalities invoving two
log-preinvex functions, J. Inequal. Pure Appl. Math., 8 (2007), No. 3, 1-6,
Article 75.

\bibitem{N07b} M. Aslam Noor, Hermite-Hadamard integral inequalities for
log-preinvex functions, J. Math. Anal. Approx. Theory, 2 (2007), 126-131.

\bibitem{BGD11b} A. Barani, A.G. Ghazanfari, S.S. Dragomir, Hermite-Hadamard
inequality for functions whose derivatives absolute values are preinvex,
RGMIA Res. Rep. Coll., 14(2011), Article 64.

\bibitem{A05} T. Antczak, Mean value in invexity analysis, Nonlinear
Analysis 60 (2005) 1471-1484.

\bibitem{YL01} X.M. Yang and D. Li, On properties of preinvex functions, J.
Math. Anal. Appl. 256 (2001) 229-241.

\bibitem{MN95} S.R.Mohan and S.K. Neogy, On invex sets and preinvex
functions, J. Math. Anal. Appl. 189 (1995), 901-908.
\end{thebibliography}
\end{document}